\documentclass{amsart}

\usepackage{amsmath}
\usepackage{amsthm}
\usepackage{amssymb}
\usepackage{latexsym}
\usepackage{graphicx}
\usepackage{tikz}
\usepackage{hyperref}

\theoremstyle{plain}
\newtheorem{theorem}{Theorem}[section]
\newtheorem{lemma}[theorem]{Lemma}

\newtheorem{corollary}[theorem]{Corollary}

\theoremstyle{definition}

\theoremstyle{remark}
\newtheorem{remark}[theorem]{Remark}

\numberwithin{equation}{section}

\def\R{\mathbb{R}}
\def\C{\mathbb{C}}

\newcommand{\comment}[1]{}

\newcommand{\eq}{\begin{equation}}
\newcommand{\en}{\end{equation}}

\def\N{\mathbb{N}}

\begin{document}
\title{Universality of Correlations for Random Analytic Functions}

\author{Shannon Starr}
\address{Department of Mathematics, University of Rochester, Rochester, NY 14627, USA.}
%\email{sstarr@math.rochester.edu}


\date{July 19, 2011}

\begin{abstract}

We review a  result obtained with Andrew Ledoan and Marco Merkli.
Consider a random analytic function $f(z) = \sum_{n=0}^{\infty} a_n X_n z^n$,
where the $X_n$'s are i.i.d., complex valued random variables with mean zero and unit variance,
and the coefficients $a_n$ are non-random and chosen so that the variance transforms covariantly
under conformal transformations of the domain.
If the $X_n$'s are Gaussian, this is called a Gaussian analytic function (GAF).
We prove that, even if the coefficients are not Gaussian, the zero set converges in distribution
to that of a GAF near the boundary of the domain.

\end{abstract}

\maketitle

\section{Introduction}

Random polynomials and random analytic functions are a topic of classical interest.
They have enjoyed renewed interest in the past decade or so.
One can imagine various motivations.
The characteristic polynomial of a random matrix is a random polynomial.
Similarly, the partition function for a Bernoulli spin glass may be viewed as a random analytic function
in the variable $z=e^{\beta}$.
This is related to the celebrated Lee-Yang theorem.

However, in mathematical study, it is easier to consider a different type of random analytic functions, where
the coefficients are independent.
When the coefficients are i.i.d., the expected number of zeros has been well-studied.
Consider for example the random analytic function
$$
f(z)\, =\, \sum_{n=0}^{\infty} X_n z^n\, ,
$$
where the coefficients are i.i.d.
Then under general conditions, the zero set accumulates at the unit circle.
A recent result \cite{IZ} has found the sharp condition for the zero
set to be asymptotically uniformly distributed on the circle.
This type of result is akin to a law of large number: for example Ibragimov and Zaporozhets
prove that the empirical distribution of the zeros converges, almost surely, to the uniform
distribution on the circle.

Various researchers have also considered the correlations between the zeros, under the assumption
that the coefficients are all Gaussian.
For instance, a recent monograph of Hough, Krishnapur, Peres and Vir\'ag reviews this topic \cite{HKPV}.
A natural question arises: if the zeros of a random analytic function accumulate near the boundary,
even under the assumption that the coefficients are not Gaussian,
what can we say about the correlations of the zeros near the boundary?
For general reasons, one expects the distribution of the zeros to converge back to the distribution of the zeros
of the Gaussian analytic function, when properly rescaled.
That is what we prove.

One would expect to have various applications of such a result.
However, some desired applications are presenly out of reach.
The polynomials corresponding to random matrices and spin glasses
generally do not have i.i.d. coefficients, except for companion matrices.
There has been some recent interest in polynomials with discrete random coefficients.
We will review this topic in the last section.

\section*{Acknowledgments}

This paper reviews research done in collaboration with Marco Merkli and Andrew Ledoan.
I am most grateful to them.

\section{Set-up}

We consider a special ensemble of random analytic functions, inspired by the monograph \cite{HKPV}.
Given a parameter $\kappa \leq 0$, and a 
sequence of coefficients
$\boldsymbol{x} = (x_0,x_1,x_2,\dots)$,
one may define the power series
$$
f_\kappa(\boldsymbol{x},z)\, =\, \sum_{n=0}^{\infty} a_{n,\kappa} x_n z^n\, ,
$$
where
$$
a_{n,\kappa}\, =\, \prod_{j=1}^{n} \left[\frac{1 -(j-1)\kappa}{j}\right]^{1/2}\, .
$$
We consider random analytic functions (RAF's) defined by choosing a coefficient
sequence $\boldsymbol{X} = (X_0,X_1,\dots)$ where $X_0, X_1, \dots$ are i.i.d.,
complex valued random variables, with mean zero and unit variance, such that
\begin{equation}
\label{eq:unitvariance}
{\bf E}\big[ (\textrm{Re}[X_i])^2\big]\, 
=\, {\bf E} \big[ (\textrm{Im}[X_i])^2\big]\, ,\quad
{\bf E} \textrm{Re}[X_i] \textrm{Im}[X_i]\, =\, 0\, .
\end{equation}
A number of important properties hold for these models, which we describe now.
An excellent reference, with complete proofs is \cite{HKPV}.

We write $\mathbb{U}(z,r)$ for the open disk $\{w\in \C\, :\, |w-z|<r\}$.
The natural domain of convergence for $f_{\kappa}(\boldsymbol{X},z)$ is $\mathbb{U}(0,\rho_\kappa)$,
a.s., where $\rho_{\kappa} = |\kappa|^{-1/2}$.
By (\ref{eq:unitvariance}), 
${\bf E}[f_{\kappa}(\boldsymbol{X},z) f_{\kappa}(\boldsymbol{X},w)]\, =\, 0$, and
\begin{equation}
\label{eq:covariance}
{\bf E}[f_{\kappa}(\boldsymbol{X},z) \overline{f_{\kappa}(\boldsymbol{X},w)}]\, =\, Q_{\kappa}(z,w)\, :=\,
\begin{cases} (1+\kappa z \overline{w})^{1/\kappa} & \text { for $\kappa\neq 0$,}\\
e^{z \overline{w}} & \text { for $\kappa = 0$.}
\end{cases}
\end{equation}
The function $Q_{\kappa}$ possesses important symmetries. For $|u|<\rho_{\kappa}$, consider the M\"obius transformation
$$
\Phi_{\kappa}^{u}(z)\, =\, \frac{z-u}{1+\kappa \overline{u} z}\, ,
$$
which is a univalent mapping of $\mathbb{U}(0,\rho_{\kappa})$ to itself.
This is an isometry relative to a metric with Gauss curvature $4\kappa$.
Moreover,
$$
Q_{\kappa}\left(\Phi_{\kappa}^u(z),\, \Phi_{\kappa}^u(w)\right)\, =\, \Delta_{\kappa}^u(z) \overline{\Delta_{\kappa}^u(w)} Q_{\kappa}(z,w)\, ,
$$
where
$$
\Delta_{\kappa}^u(z)\, 
=\, \begin{cases}
(1 + \kappa |u|^2)^{1/(2\kappa)} (1+\kappa \overline{u} z)^{-1/\kappa} & \text { for $\kappa\neq 0$,}\\
\exp(\frac{1}{2} |u|^2 - \overline{u} z) & \text { for $\kappa = 0$.}
\end{cases}
$$
While $Q_{\kappa}$ is not {\em invariant} with respect to the isometries $\Phi_{\kappa}^u$,
one says it is {\em covariant} because of this property.
Also note that for any $z \in \mathbb{U}(0,\rho_{\kappa})$, 
$$
|\Phi_{\kappa}^u(z)| \to \rho_{\kappa}\quad \text { as }\quad |u| \to \rho_{\kappa}\, .
$$
Taking $u$ to the boundary of the domain $\mathbb{U}(0,\rho_{\kappa})$,
conformally maps neighborhoods of $0$ to domains approaching the boundary.

Gaussian analytic functions (GAF's) are important special cases of RAF's.
Their zero sets have been studied in \cite{HKPV}, and many interesting questions
about these zero sets continue to be studied.
The reader may consult that reference and references therein.
Our main result proves convergence in distribution of the zero sets of the RAF's,
for a sequence of neighborhoods converging to the boundary.

\begin{theorem}[Main Result]
\label{thm:main}
Suppose that $X_0,X_1,\dots$ are i.i.d., complex-valued random variables with mean zero
and satisfying (\ref{eq:unitvariance}). 
Let $\boldsymbol{Z} = (Z_0,Z_1,\dots)$ be i.i.d., complex Gaussians with density $\pi^{-1} e^{-|z|^2/2}$
on the complex plane.
For each $\kappa\leq 0$, and any continuous function $\varphi$ whose support is a compact subset of $\mathbb{U}(0,\rho_{\kappa})$,
the random variables
$$
\sum_{\xi\, :\, f_{\kappa}(\boldsymbol{X},\xi)=0} \varphi(\Phi_{\kappa}^{u}(\xi))
$$
converge in distribution, in the limit $|u| \to \rho_{\kappa}$, to the random variable
$$
\sum_{\xi\, :\, f_{\kappa}(\boldsymbol{Z},\xi)=0} \varphi(\xi)\, .
$$
\end{theorem}

\begin{figure}
$$
\begin{minipage}{4cm}
\begin{tikzpicture}
\draw (0,0) circle (2cm);
\draw[dashed] (0.85,0.85) circle (0.5cm);
\fill (0.51,0.886) circle (0.75mm);
\fill (0.987,0.712) circle (0.75mm);
\fill (0.63,1.173) circle (0.75mm);
\draw[->,line width=1mm] (0.85,0.85) .. controls +(-0.2,0) and +(0.1,0.2) .. (0,0);
\fill[white] (0.85,0.85) circle (1.5mm);
\draw (0.85,0.85) node[] {$u$};
\end{tikzpicture}
\end{minipage}
\qquad 
\begin{minipage}{1.5cm}
\begin{tikzpicture}
\draw (0,0) node {\Large $\stackrel{\Phi_{\kappa}^u}{\boldsymbol{\longrightarrow}}$};
\end{tikzpicture}
\end{minipage}
\qquad
\begin{minipage}{4cm}
\begin{tikzpicture}
\draw (0,0) circle (2cm);
\draw[dashed] (0,0) circle (1cm);
\begin{scope}[xscale=2,yscale=2,xshift=-0.85cm,yshift=-0.85cm]
\fill (0.51,0.886) circle (0.5mm);
\fill (0.987,0.712) circle (0.5mm);
\fill (0.63,1.173) circle (0.5mm);
\end{scope}
\draw (0,0) node {$0$};
\end{tikzpicture}
\end{minipage}
$$
\caption{\label{fig:1}
To study the zeroes in a neighborhood of a point $u$ near the boundary of $\mathbb{U}(0,\rho_{\kappa})$, take the  image under the map $\Phi_{\kappa}^u$,
which maps $u$ to $0$,
and consider the positions of the zeroes under this mapping.
This is described in Remark \ref{rem:1}.
}
\end{figure}
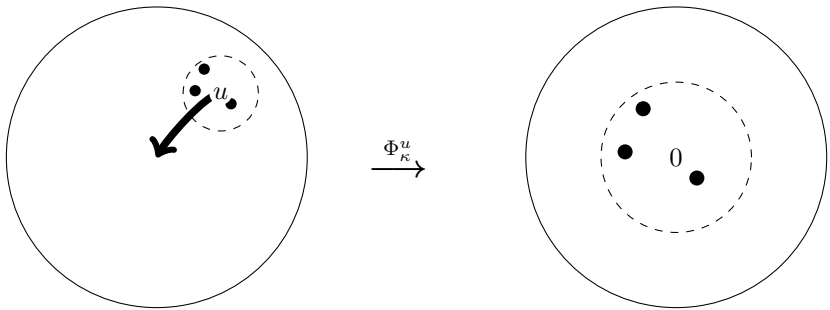

Since $Q_{\kappa}$ is covariant with respect to the mappings $\Phi_{\kappa}^u$,
and since the distribution of a Gaussian process is determined by its covariance, the zeros of the GAF, $\{\xi\, :\, f_{\kappa}(\boldsymbol{Z},\xi)=0\}$
is a stationary point process with respect to these mappings.

\begin{remark} \label{rem:1}
The mapping $\Phi_{\kappa}^u$ was defined so that $\Phi_{\kappa}^u(u)=0$.
Therefore, mapping the zeroes by $\Phi_{\kappa}^u$ maps the zeroes in a neighborhood
of $u$ to a neighborhood of $0$.
We have an illustration in Figure \ref{fig:1} to indicate this.
The test function $\varphi$ is nonzero only in a window around $0$.
\end{remark}

There is some interest in also considering random polynomials with real coefficients.
This motivates considering random real analytic functions.

\begin{theorem}
\label{thm:second}
Suppose that $X_0,X_1,\dots$ are i.i.d., real-valued random variables with mean zero
and variance 1.
Let $Z_0,Z_1,\dots$ be i.i.d., real-valued standard, normal random variables.
For each $\kappa\leq 0$, and any continuous function $\varphi$ whose support is a compact subset of $\mathbb{U}(0,\rho_{\kappa})$,
$$
\sum_{\xi\, :\, f_{\kappa}(\boldsymbol{X},\xi)=0} \varphi(\Phi_{\kappa}^{r}(\xi))
\quad \text {converges in distribution to} \quad 
\sum_{\xi\, :\, f_{\kappa}(\boldsymbol{Z},\xi)=0} \varphi(\xi)\, ,
$$
in the limit $|r| \to \rho_{\kappa}$, along any sequence satisfying the constraint $r \in \R$.
\end{theorem}

\section{Proof of the Main Result}

The proof is elementary.
It uses several tools from probability theory and complex analysis such as the Lindeberg-Feller
condition and Hurwitz's theorem.
Since these are well-known to probabilists, these results were merely referred to in an implicit
way in the version of our paper \cite{LMS} in order to shorten the presentation.
But here, we will also briefly review those tools.
(In an earlier version of our paper, which is available in preprint form, we did also
include these tools.)

We transcribe  the Lindeberg-Feller version of the central limit theorem
from Durrett's textbook \cite{Durrett}.
(See page 110 in the latest, online version.)
\begin{theorem}[Lindeberg Feller Theorem]
Suppose that for each $n$, there is a sequence of real random variables $(X_{n,m})$
such that for a fixed $n$ these random variables are independent (for different
$m$ indices) and ${\bf E}X_{n,m}=0$ for all $m$. Suppose
\begin{itemize}
\item[(i)] $\sum_{m=1}^{\infty} {\bf E} X_{n,m}^2 \to \sigma^2 > 0$ as $n \to \infty$,
\item[(ii)] For all $\epsilon>0$, $\lim_{n \to \infty} {\bf E}(|X_{n,m}|^2 ; |X_{n,m}|>\epsilon) = 0$.
\end{itemize}
Then $S_n = \sum_{m=1}^{\infty} X_{n,m}$ converges in distribution to 
$\sigma \chi$, where $\chi$ is a standard normal (real) random variable.
\end{theorem}
Durrett also lists an exercise in his textbook to derive ``Lyapunov's theorem''
from this. (It is Exercise 3.4.12 in the most recent version.)
Let us also state this corollary.

\begin{corollary}
\label{cor:moment}
Suppose that $X_0,X_1,\dots$ are i.i.d., complex valued random variables with mean zero,
satisfying condition (\ref{eq:unitvariance}).
Suppose that $(\alpha_{n,k})$ is a complex valued set of numbers, for $n \in \N$
and $k \in \{0,1,\dots \}$,
satisfying
\begin{itemize}
\item[(a)] $\frac{1}{2} \sum_{k=0}^{\infty} |\alpha_{n,k}|^2 \to \sigma^2 > 0$ as $n \to \infty$, and
\item[(b)] $\sum_{k=0}^{\infty} |\alpha_{n,k}|^p \to 0$ as $n \to \infty$ for some $p\geq 0$.
\end{itemize}
Then $S_n = \sum_{k=1}^{\infty} \textrm{Re}[\alpha_{n,k} X_k]$ converges in distribution to 
$\sigma \chi$.
\end{corollary}

We will not prove these results, which are available in \cite{Durrett}.
However, we will use Lyapunov's condition for the Lindeberg-Feller theorem
to prove:

\begin{lemma}
\label{lem:clt}
Suppose that $X_0,X_1,\dots$ are i.i.d., complex-valued random variables with mean zero
and satisfying (\ref{eq:unitvariance}). 
Let $\boldsymbol{Z} = (Z_0,Z_1,\dots)$ be i.i.d., complex Gaussians with density $\pi^{-1} e^{-|z|^2/2}$
on the complex plane.
Then for any $N \in \N$, any $z_1,\dots,z_N \in \mathbb{U}(0,\rho_{\kappa})$ and any $\lambda_1,\dots,\lambda_N \in \C$, the random variables 
$$
\sum_{k=1}^{N} \lambda_k \frac{f_{\kappa}(\boldsymbol{X},\Phi_{\kappa}^u(z_k))}{\Delta_{\kappa}^u(z_k)}
$$
converge in distribution, to the random variable $\sum_{k=1}^{N} \lambda_k f_{\kappa}(\boldsymbol{Z},z_k)$,
as $|u| \to \rho_{\kappa}$,
\end{lemma}
%This is proved using the following simple corollary of the Lindeberg-Feller central limit theorem. 
%
%\begin{cor}
%\label{cor:LF}
%Suppose that $X_0,X_1,\dots$ are i.i.d., complex valued random variables with mean zero,
%satisfying condition (\ref{eq:unitvariance}).
%Suppose that $\alpha_{k}(u) \in \C$ is defined for each $u \in U(0,1)$
%and $k \in \{0,1,\dots \}$,
%satisfying
%\begin{itemize}
%\item[(a)] $\frac{1}{2} \sum_{k=0}^{\infty} |\alpha_{k}(u)|^2 \to \sigma^2 > 0$ as $|u| \uparrow 1$, and
%\item[(b)] $\sum_{k=0}^{\infty} |\alpha_{k}(u)|^p \to 0$ as $|u| \uparrow 1$, for some
%$p>2$.
%\end{itemize}
%Then $\sum_{k=1}^{\infty} \textrm{Re}[\alpha_{k}(u) X_k]$ converges in distribution to 
%$\sigma \chi$ as $|u| \uparrow 1$, where $\chi$ is a real standard normal  random variable.
%\end{cor}
%See, for example, \cite{Durrett} for the Lindeberg-Feller theorem.
%With this, we can prove the lemma.

\begin{proof}
We may write
$$
\sum_{k=1}^{N} \lambda_k \frac{f_{\kappa}(\boldsymbol{X},\Phi_{\kappa}^u(z_k))}{\Delta_{\kappa}^u(z_k)}\,
=\, \sum_{k=1}^{N} \frac{\lambda_k}{\Delta_{\kappa}^u(z_k)} \sum_{n=0}^{\infty} a_{n,\kappa} X_n \left(\Phi_{\kappa}^u(z_k)\right)^n\,
=\, \sum_{n=0}^{\infty} \alpha_{n,\kappa}(u) X_n\, ,
$$
where we have left the dependence of $\alpha_{n,\kappa}(u)$ on $\lambda_1,\dots,\lambda_n$ and $z_1,\dots,z_n$ implicit for the coefficients
$$
\alpha_{n,\kappa}(u)\, =\, a_{n,\kappa} \sum_{k=1}^{N} \frac{\lambda_k \left(\Phi_{\kappa}^u(z_k)\right)^n}{\Delta_{\kappa}^u(z_k)}\, .
$$
Since $Q_{\kappa}$ is {\em covariant} with respect to the transformations $\Phi_{\kappa}^u$, this implies that
$$
\sum_{n=0}^{\infty} |\alpha_{n,\kappa}(u)|^2\,
=\, \sum_{j,k=1}^{N} \lambda_j Q_{\kappa}(z_j,z_k) \overline{\lambda}_k\, .
$$
For the same reason the variance of the random variables in question is a constant function of $u$.

To apply the Lindeberg-Feller conditions, we need to check Lyapunov's condition for the sequence $\alpha_{n,k}(u)$.
We note that for $p=4$
$$
\sum_{n=0}^{\infty} |\alpha_{n,\kappa}(u)|^4\,
\leq \, N^4 \max_{k=1,\dots,N} |\lambda_k|^4 
\sum_{n=0}^{\infty} a_{n,\kappa}^4 \frac{|\Phi_{\kappa}^u(z_k)|^{4n}}{|\Delta_{\kappa}^u(z_k)|^4}\, .
$$
Cauchy's integral formula implies
$$
\sum_{n=0}^{\infty} a_{n,\kappa}^4 \frac{|\Phi_{\kappa}^u(z)|^{4n}}{|\Delta_{\kappa}^u(z)|^4}\,
=\, \frac{1}{2\pi i} \oint_{C(0,1)} \left|\sum_{n=0}^{\infty} a_{n,\kappa}^2 \frac{|\Phi_{\kappa}^u(z)|^{2n}}{|\Delta_{\kappa}^u(z)|^2}\, \zeta^n\right|^2\,  d\zeta\, ,
$$
for each fixed $z \in \mathbb{U}(0,\rho_{\kappa})$.
For $\kappa<0$ the series sums to
$$
\sum_{n=0}^{\infty} a_{n,\kappa}^2 \frac{|\Phi_{\kappa}^u(z)|^{2n}}{|\Delta_{\kappa}^u(z)|^2}\, \zeta^n\,
=\, 
(1+\kappa |z|^2)^{1/\kappa} \left(\frac{1+\kappa |\Phi_{\kappa}^u(z)|^2 \zeta}{1+\kappa |\Phi_{\kappa}^u(z)|^2}\right)^{1/\kappa}\, .
$$
The second factor on the right hand side has norm bounded by 1 for all $\zeta \in C(0,1)$.
Moreover since $|\Phi_{\kappa}^u(z)|$ converges to $\rho_{\kappa} = |\kappa|^{-1/2}$ in the limit $|u| \to \rho_{\kappa}$,
the second factor converges pointwise to $0$ in that limit,  for every $\zeta \in C(0,1) \setminus \{1\}$.
For $\kappa=0$ the series sums to
$$
e^{|z|^2} \exp\left((\zeta-1)|\Phi_{\kappa}^u(z)|^2\right)\, .
$$
But for $\kappa=0$, we know $\rho_0=\infty$ and $|\Phi_{0}^u(z)| \to \infty$ in the limit $|u| \to \infty$.
Since the real part of $(\zeta-1)$ is non-positive, the same conclusion follows.
In either case, the dominated convergence gives the desired result.
\end{proof}

Lemma \ref{lem:clt} implies that the random analytic functions
$[\Delta_{\kappa}^u(z)]^{-1} f_{\kappa}(\boldsymbol{X},\Phi_{\kappa}^u(z))$ converge in distribution to the random analytic function $f_{\kappa}(\boldsymbol{Z},z)$,
in the limit $|u| \to \rho_{\kappa}$, in the sense that the finite dimensional marginals of the function values converge.
This also implies convergence in distribution of the zero sets.
We will write $\Rightarrow$ for convergence in distribution. 
A clear and elegant proof of this fact has been provided by Valko and Vir\'ag in a recent paper
they wrote on random Schr\"odinger operators \cite{ValkoVirag}.

\begin{lemma}[Valko and Vir\'ag, 2010]
\label{lem:ValkoVirag}
Let $f_n(\omega,z)$ be a sequence of random analytic functions on a domain $D$ (which is open, connected and simply connected)
such that ${\bf E} h(|f_n(z)|) < g(z)$ for some increasing unbounded function $h$
and a locally bounded function $g$. Assume that $f_n(z) \Rightarrow f(z)$ in the sense of 
finite dimensional distributions. Then $f$ has a unique analytic version and 
$f_n \Rightarrow f$ in distribution with respect to local-uniform convergence. 
\end{lemma}

Because of this result we see that $[\Delta_{\kappa}^u(z)]^{-1} f_{\kappa}(\boldsymbol{X},\Phi_{\kappa}^u(z))$ converges
in distribution to $f_{\kappa}(\boldsymbol{Z},z)$, with respect to the local uniform convergence.
To converge in distribution
with respect to local-uniform convergence means that for any function $\mathcal{F}$ of $f$ which is continuous
with respect to the local-uniform topology, the random variables $\mathcal{F}(f_n)$
converge in distribution to $\mathcal{F}(f)$.
If $\varphi$ is a continuous function, compactly supported on $\mathbb{U}(0,\rho_{\kappa})$,
then defining the function
$$
\mathcal{N}_{\varphi}(f)\, 
=\, \sum_{\xi\, :\, f(\xi)=0} \varphi(\xi)\, ,
$$
this is continuous with respect to the local-uniform topology.
For the reader's convenience we paraphrase Hurwitz's theorem from Saks and Zygmund \cite{SaksZygmund},
page 158.

\begin{theorem}[Hurwitz's theorem]
If a sequence $(h_n(z))$ of functions, continuous on a closed set 
$\mathcal{K}$ and holomorphic in the interior of $\mathcal{K}$,
is uniformly convergent on this set, and if the function
$h(z) = \lim_{n \to \infty} h_n(z)$ vanishes nowhere on the boundary
of the set $\mathcal{K}$, then, beginning from a certain value of $n$,
all the functions $h_n(z)$ have in the interior of $\mathcal{K}$
the same number of roots as the function $h(z)$ (counting every root 
as many times as its multiplicity indicates).
\end{theorem}

From this theorem, it is easy to see that $\mathcal{N}_{\varphi}$ is a continuous function.
This follows from the usual method of approximation by simple functions based on disks.

Note that $\Delta_{\kappa}^u(z)$ is finite and non-vanishing for $z \in \mathbb{U}(0,\rho_{\kappa})$.
Therefore 
$$
\mathcal{N}_{\varphi}\left( [\Delta_{\kappa}^u(z)]^{-1} f_{\kappa}(\boldsymbol{X},\Phi_{\kappa}^u(z)) \right)
\, =\, \mathcal{N}_{\varphi}\left( f_{\kappa}(\boldsymbol{X},\Phi_{\kappa}^u(z)) \right)\, .
$$
Moreover, $f_{\kappa}(\boldsymbol{X},\Phi_{\kappa}^u(z))=0$ means
that $z = \Phi_{\kappa}^{-u}(w)$ for some $w$ such that $f_{\kappa}(\boldsymbol{X},w)=0$.
Therefore, we see that convergence of the distributions of
$\mathcal{N}_{\varphi}\left( f_{\kappa}(\boldsymbol{X},\Phi_{\kappa}^u(z)) \right)$ really means
that
$$
\sum_{\xi\, :\, f_{\kappa}(\boldsymbol{X},\xi)=0} \varphi(\Phi_{\kappa}^{u}(\xi))
\quad \text {converges in distribution to} \quad 
\sum_{\xi\, :\, f_{\kappa}(\boldsymbol{Z},\xi)=0} \varphi(\xi)\, ,
$$
in the limit $|u| \to \rho_{\kappa}$, as claimed.

\section{An Application}

Recently there has been some interest in looking at {\em all}
random polynomials of a given degree with coefficients in $\{+1,-1\}$.
John Baez reported on this on his blog,
and there was a popular article on the topic with numerical results
obtained by Sam Derbyshire \cite{Baez}.
We can write such polynomials as
$$
p_n(X_0,\dots,X_n;z)\, =\, \sum_{k=0}^n X_k z^k\, ,\quad 
\text { where } \quad 
X_0,\dots,X_n \in \{+1,-1\}\, .
$$
In Figure \ref{fig:roots}, we have plotted a representation of the zeros for $n=13$.
For any fixed sequence $X_0,X_1,\dots \in \{+1,-1\}$, it is easy to see that
$$
p_n(X_0,\dots,X_n;z) \to f_{-1}(\boldsymbol{X},z)\, ,\quad \text { locally, uniformly}
$$
as $n \to \infty$.
In particular the zero sets converge.

\begin{figure}[]
$$
\begin{tikzpicture}[xscale=0.66,yscale=0.66]
\draw (0,0) node[above] {\includegraphics[height=4.1cm,width=4cm]{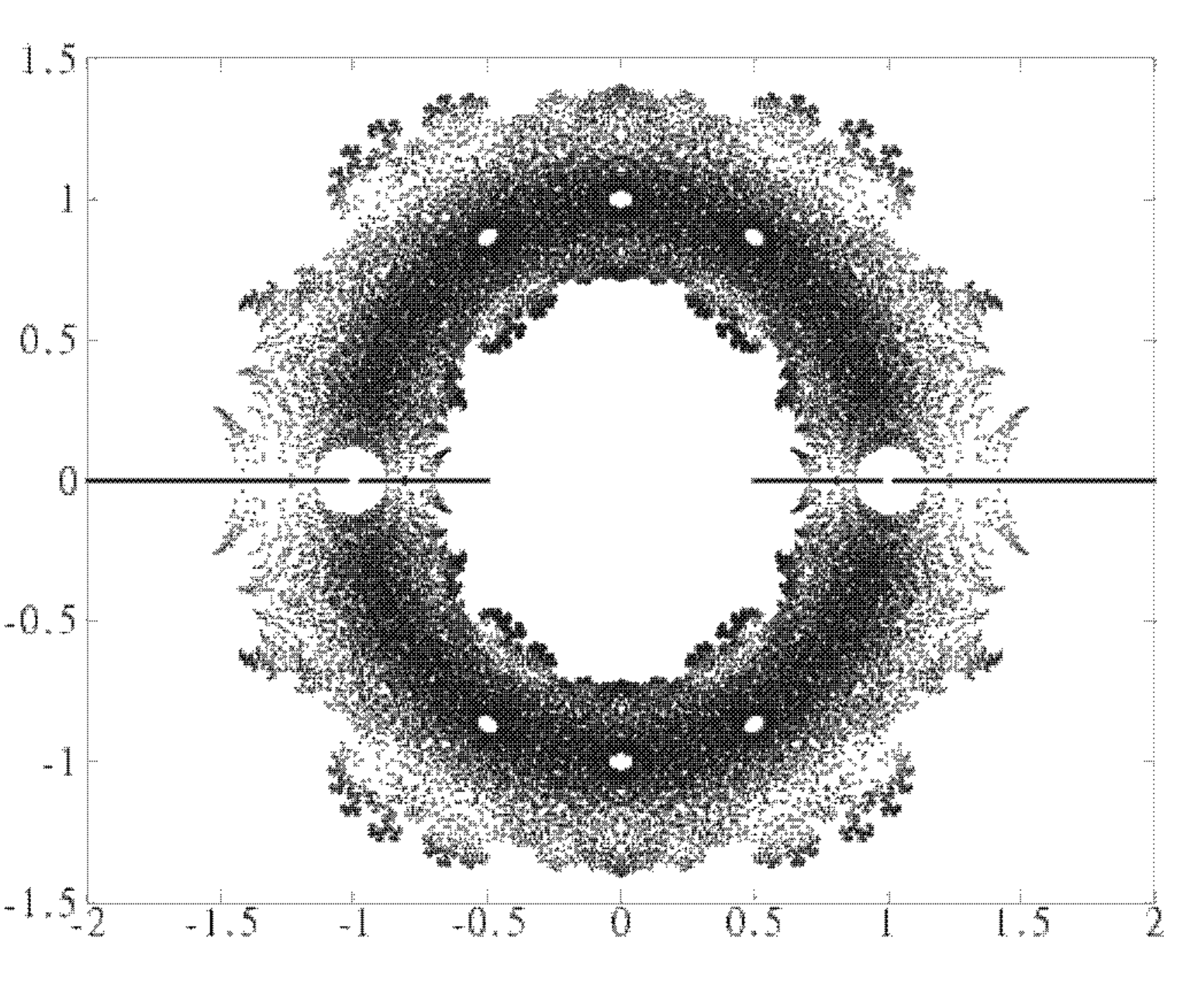}};
\draw (0,0) node[below] {(a) $\mathcal{Z}_{13}$};
\draw (7,0) node[above] {\includegraphics[height=4cm,width=4cm]{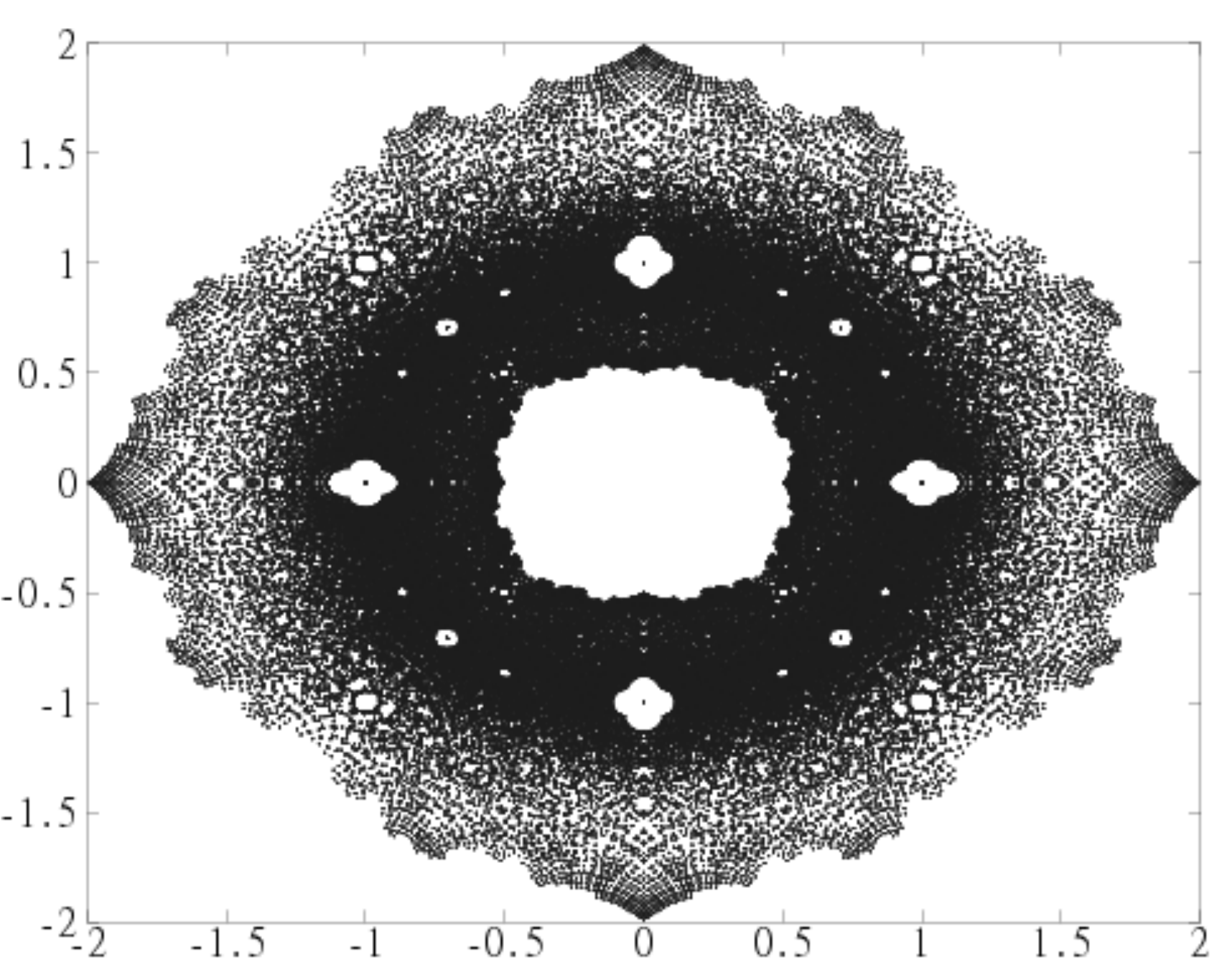}};
\draw (7,0) node[below] {(b) $\mathcal{W}_8$};
\end{tikzpicture}
$$
\caption{\label{fig:roots}
(a) This is the set of all $z \in \mathbb{C}$, such that $p_n(X_0,\dots,X_n;z)=0$
for some choice of $X_0,\dots,X_n \in \{+1,-1\}$, for $n=13$.
(b) The set of all roots for polynomials with coefficient $X_0,\dots,X_n \in \{1+i,1-i,-1+i,-1-i\}$
for $n=8$.
(The roots are computed numerically using Octave's root program. Better figures are available on
the arXiv version of this paper.)}
\end{figure}

Note that since the coefficients are strictly real, our main theorem does not apply.
But this is why we stated Theorem \ref{thm:second}.
The proof of Theorem \ref{thm:second} is identical to that of Theorem \ref{thm:main}.
The only difference is that the covariance changes.
Instead of having $E[f(\boldsymbol{X},z) f(\boldsymbol{X},w)]=0$, we have
$$
\overline{f(\boldsymbol{X},z)} = f(\boldsymbol{X},\overline{z})\, ,
$$
because all coefficients are real.
The M\"obius transformations $\Phi^r_{\kappa}$ for $r \in \R$ also preserve this property.
That is why we restricted to those isometries in the statement of Theorem \ref{thm:second}.

One can also consider the polynomials $p_n(X_0,\dots,X_n;z)$ where $X_0,\dots,X_n$ are i.i.d.,
uniform in the set $\{1+i,1-i,-1+i,-1-i\}$.
In this case, Theorem \ref{thm:main} applies.
The simplest non-Gaussian RAF is given by
$$
f_{-1}(\boldsymbol{X},z)\, =\, \sum_{n=0}^{\infty} X_n z^n\, ,
$$
where $X_0,X_1,\dots$ are i.i.d., random variable chosen from the set $\{1,-1,i,-i\}$ with equal probabilities.
Our theorem implies that the zero set of such a RAF near the unit circle is has a distribution which is close to that
of the corresponding GAF.

Let us define
$$
\mathcal{Z}_n\, =\, \bigcup_{X_0,\dots,X_n \in \{+1,-1\}} \{ z \in \C \, :\, p_n(X_0,\dots,X_n;z)=0\}\, ,
$$
and
$$
\mathcal{W}_n\, =\, \bigcup_{X_0,\dots,X_n \in \{1+i,1-i,-1+i,-1-i\}} \{ z \in \C \, :\, p_n(X_0,\dots,X_n;z)=0\}\, .
$$
Baez notes several ``holes'' in the sets $\mathcal{Z}_n$ centered at points on the unit circle, such as $1$
and $-1$.
From Figure \ref{fig:roots}, one also sees several holes in the set $\mathcal{W}_n$ for $n=8$
along the unit circle.
Due to our theorem, we can deduce that the holes in $\mathcal{Z}_n$ at $\pm 1$ must close up.
Similarly, all the holes in the set $\mathcal{W}_n$ along the unit circle must close in the limit 
$n \to \infty$.

For $Z_0,Z_1,\dots$ i.i.d., real-valued standard, normal random variables,
the zeros of $f_{-1}(\boldsymbol{Z},z)$ are asymptotically uniformly distributed along the unit circle
due to the result of Ibragimov and Zaporozhets, for example.
The intensity measure of $f_{-1}(\boldsymbol{X},z)$ is asymptotically close to that of $f_{-1}(\boldsymbol{Z},z)$, in the weak topology, near $\pm 1$.
Therefore, there cannot be a hole at $1$ or $-1$.
This is the simplest consequence of Theorem \ref{thm:second}.
Similarly, in the complex case.

One can also consider this from another perspective.
For a fixed value of $z$, we may define the set of function values for all possible coefficients
$$
C_n(z)\, =\, \{p_n((x_0,\dots,x_n),z)\, :\, x_0,\dots,x_n \in \{1,-1\}\}\, .
$$
This satisfies the recurrence  relation
$$
C_n(z)\, =\, \{1+z w\, :\, w \in C_{n-1}(z)\} \cup \{-1+zw\, :\, w \in C_{n-1}(z)\}\, .
$$
Defining the set
$$
C(z)\, =\, \{f_{-1}(\boldsymbol{y},z)\, :\, y_0,y_1,\dots \in \{1,-1\}\}\, ,
$$
one can see that $C(z) = \{1+zw\, :\, w \in C(z)\} \cup \{-1+zw\, :\, w \in C(z)\}$.
This implies the Hausdorff dimension satisfies the bound
$$
\operatorname{dim} C(z)\, \leq\, \frac{\log(2)}{\log(1/|z|)}\, .
$$
For $r \in \R$, we know that $C(r) \subseteq \R$.
Therefore, it seems reasonable to conjecture that 
$$
\operatorname{dim} C(z)\, =\, 
\begin{cases}
\min\{2,\log(2)/\log(1/|z|)\} & \text { for $z \in \C\setminus\R$,}\\
\min\{1,\log(2)/\log(1/|z|)\} & \text { for $z \in \R$.}
\end{cases}
$$
It is easy to see that $C(1/3)$ is the middle-thirds Cantor set;
whereas, $C(i/\sqrt{2})$ is the rectangle $\{x+iy\, :\, x \in [-2,2],\, y \in [-\sqrt{2},\sqrt{2}]\}$.
In Figure \ref{fig:2} (a) we have displayed $C_{15}(z)$ for $z=e^{i\pi/4}/\sqrt{2}$.

The fractal dimension of $C(z)$ may pertain to Baez's conjectures that the zero sets $\{z\, :\, 0 \in C_n(z)\}$
have a multi-fractal structure associated to $z$.
In analogy to the dimension of varieties of smooth curves,
one might guess that ``typically'' if $z$ satisfies $C(z) \ni 0$ then the set of $w$'s in a neighborhood
of $z$ with $C(w) \ni 0$ has dimension ``approximately'' equal to $\dim C(z)$.

\begin{figure}
$$
\begin{tikzpicture}
\draw (0,0) node[] {\includegraphics[height=5cm, width=5cm]{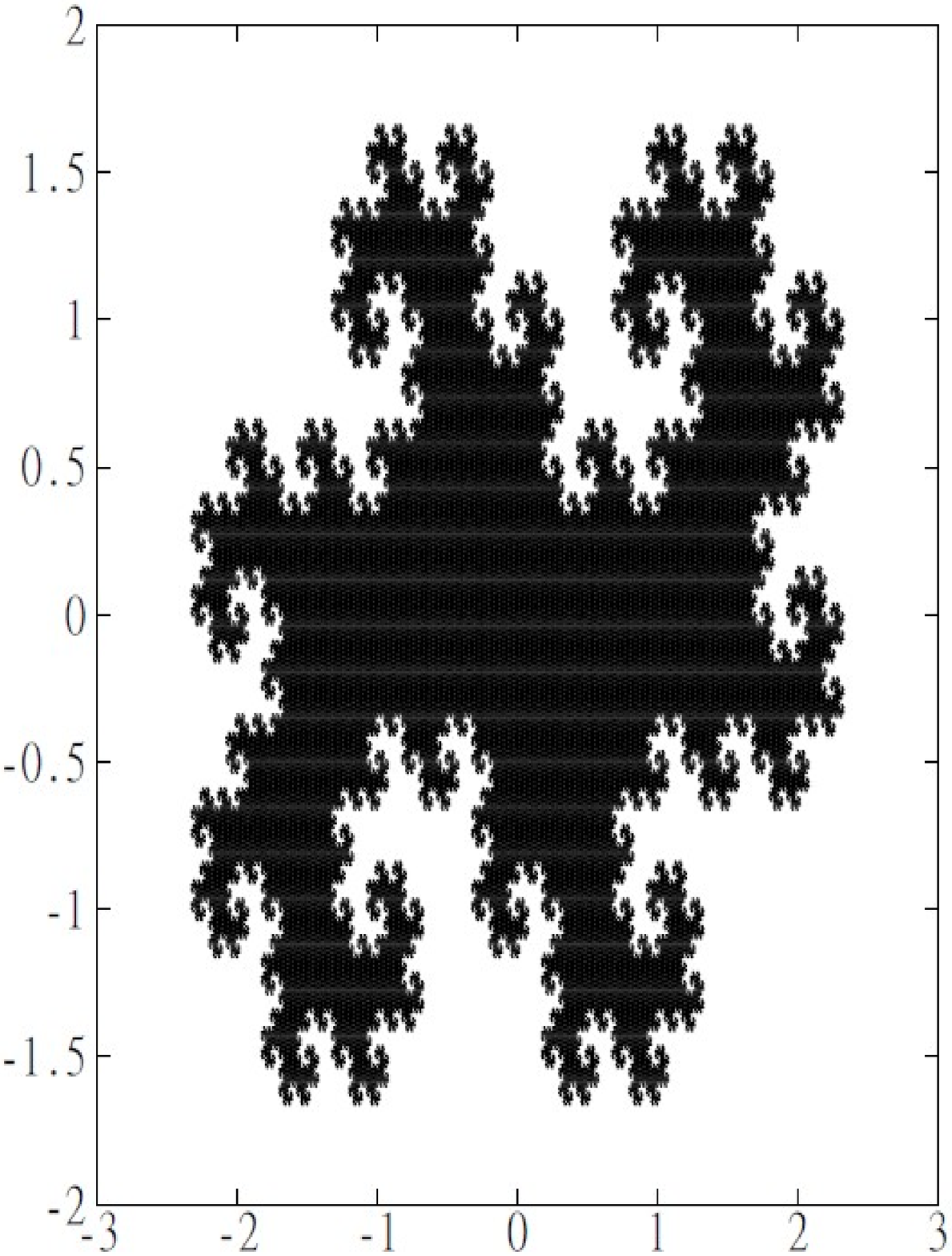}};
\draw (0,-3) node[below] {(a)};
\end{tikzpicture}
\hspace{0.75cm}
\begin{tikzpicture}
\draw (0,0) node[] {\includegraphics[height=5cm, width=5cm]{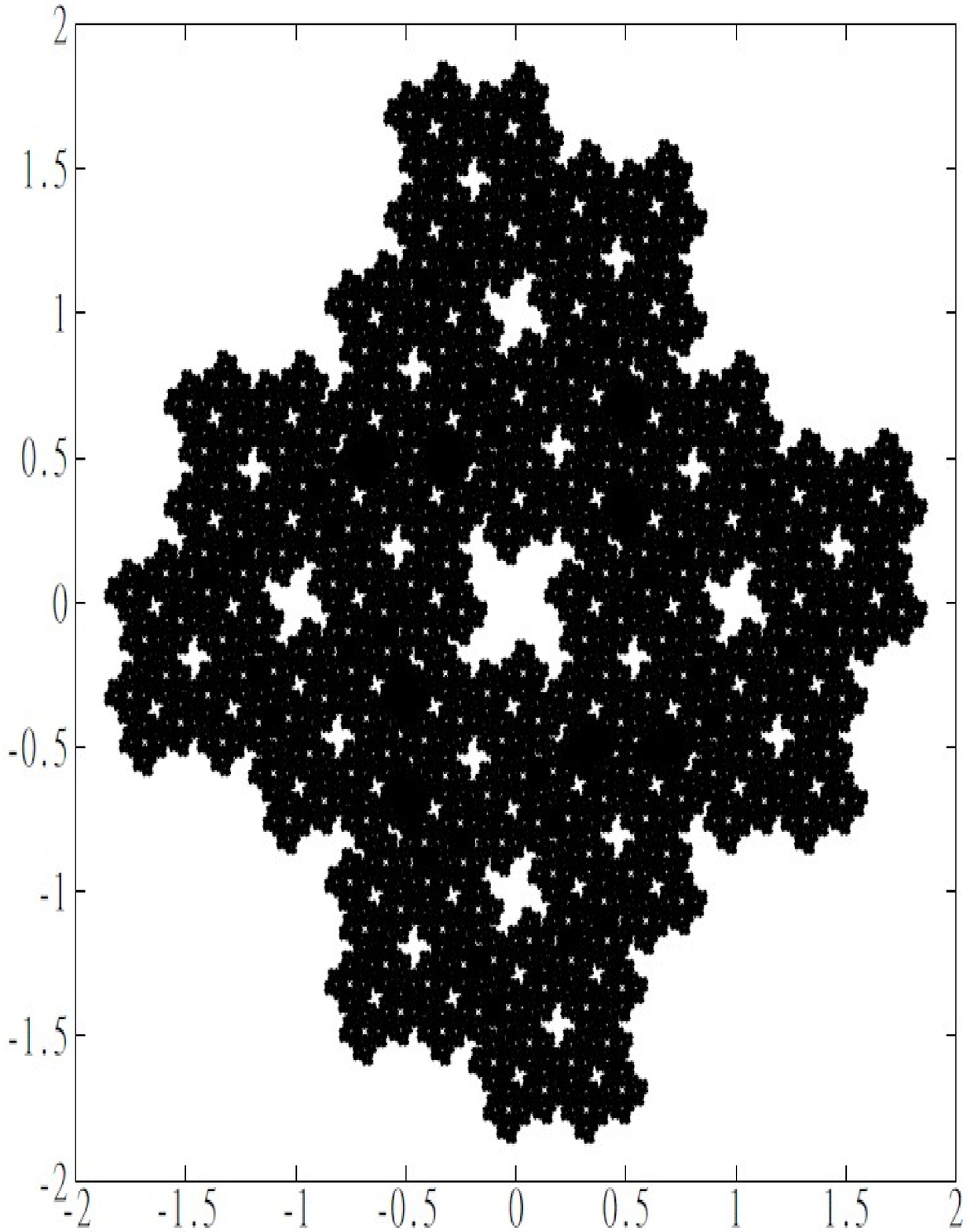}};
\draw (0,-3) node[below] {(b)};
\end{tikzpicture}
$$
\caption{\label{fig:2}
(a) The set $C_{15}(z)$ for $z=e^{i\pi/4}/\sqrt{2}$.
(b) The set $B_8(z)$ for $z = e^{i\pi/8}\sqrt{2}$.
(Better figures are available on the arXiv version of this paper.)}
\end{figure}

Similarly, defining
$$
B_n(z)\, =\, \left\{p_n((x_1,\dots,x_n),z)\, :\, x_0,x_1,\dots \in \{1,-1,i,-i\}\right\}\, ,
$$
we have $B_n(z) = \bigcup_{u \in \{1,-1,i,-i\}} \{u+zw\, :\, w \in B_{n-1}(z)\}$.
Defining 
$$
B(z)\, =\, \{f_{-1}(\boldsymbol{x},z)\, :\, x_0,x_1,\dots \in \{1,-1,i,-i\}\}\, ,
$$
this implies that $\operatorname{dim} B(z) \leq 2 \log(2)/\log(1/|z|)$.
It seems reasonable to guess that $\operatorname{dim} B(z) = \min\{2,2\log(2)/\log(1/|z|)\}$.
An easy calculation shows $B(1/2) = \{x+iy\, :\, x,y \in [-2,2]\}$.
In Figure \ref{fig:2} (b) we have plotted $B_8(z)$ for $z = e^{i\pi/8}/2$.

One can also ask, for $z$ such that $B(z) \ni 0$,
about the structure of the coefficients $\{(X_0,X_1,\dots)\, :\, X_0,X_1,\dots \in \{1,-1,i,-i\}\}$
satisfying $f_{-1}(\boldsymbol{X},z)=0$.
One guess is that as $|z| \to 1$, the distribution converges in some sense to ``uniform'' with the the density
appropriate for the intensity measure of the zeroes at that point.
This guess is affirmed for GAF's.
It is reasonable to conjecture that this is also true for RAF's with discrete coefficients.
An interesting possibility is to relate the correlations of zeros for $f_{-1}(\boldsymbol{X},z)$
for a typical value of $\boldsymbol{X}$,
to the structure of the coefficients $\boldsymbol{X}$ such that $f_{-1}(\boldsymbol{X},z)=0$
for a typical value of $z$.
This may be a topic for further study.

\bibliographystyle{amsplain}

\end{document}